\documentclass[12pt,a4paper]{amsart}
\usepackage{amsfonts}
\usepackage{amsthm}
\usepackage{amsmath}
\usepackage{amscd}
\usepackage[latin2]{inputenc}
\usepackage{t1enc}
\usepackage[mathscr]{eucal}
\usepackage{indentfirst}
\usepackage{graphicx}
\usepackage{graphics}
\usepackage{pict2e}
\usepackage{epic}
\usepackage{amssymb}
\usepackage[colorlinks,citecolor=red,urlcolor=blue,bookmarks=false,hypertexnames=true]{hyperref} 
\numberwithin{equation}{section}
\usepackage[margin=3cm]{geometry}

\theoremstyle{plain}
\newtheorem{Th}{Theorem}[section]
\newtheorem{Lemma}[Th]{Lemma}
\newtheorem{Cor}[Th]{Corollary}
\newtheorem{Prop}[Th]{Proposition}
\newtheorem*{Theorem-non}{Theorem}
\newtheorem*{Theorem-non2}{Theorem}

 \theoremstyle{definition}
 \newtheorem*{Proof-non}{Proof of Theorem \ref{Maintheorem} assuming Propositions \ref{Prop1},\ref{Propm}}
\newtheorem*{Proof-non2}{Proof of (1)  ($\bf{m_{1}}$-estimate) in Proposition \ref{Propm} assuming Proposition \ref{Proposition 5.1}}
\newtheorem*{Proof-non3}{Proof of Theorem \ref{Maintheorem2} assuming Propositions \ref{Prop1},\ref{Propm}}
\newtheorem*{Proof-non4}{Proof of Proposition \ref{Prop1}}
\newtheorem*{Proof-non5}{Proof of Proposition \ref{Propm}}

\newtheorem{Rem}[Th]{Remark}
\newtheorem{?}[Th]{Problem}

\renewcommand\rightmark{Divisor function over integers with a missing digit}

\begin{document}

\author{Jiseong Kim}
\address{The University of Mississippi, Department of Mathematics
Hume Hall 335
Oxford, MS 38677}
\email{Jkim51@olemiss.edu}
\title{The divisor function over integers with a missing digit}

\begin{abstract} 
In this paper, we study the sum of the divisor function over short intervals under digit restrictions.

\end{abstract}

\maketitle
\section{Introduction} Let X be a power of a natural number $g$ and let $[X,X+H]^{\ast}$ be the set of natural numbers $X < n \leq X+H$ such that $n$ has only digits of $\{0,1,2,3, \cdots, g-1\}\textbackslash \{b\}$ in its base-$g$ expansion, and for some $b \in \{2,3, \cdots, g-1\}.$ The size of $[X,X+H]^{\ast}$ is approximately $H^{\log (g - 1) / \log g} $ when $X$ is a power of $g.$  Due to the lack of multiplicative structure among the elements in $[X,X+H]^{\ast},$ some standard approaches for considering summations of arithmetic functions over $[X,X+H]^{\ast}$ do not work. However, because $[X,X+H]^{\ast}$ has a nice Fourier transform, various interesting results have been proved (see \cite{Erdos1}, \cite{Erdos2}, \cite{Konyagin}, \cite{Banks}, \cite{Mau}). More recently, Maynard \cite{Jmay} proved that there are infinitely many primes with restricted digits, using various sieve methods and Fourier analysis. He established very strong upper bounds for
$\sum_{n \in [1,X]^{\ast}} e(\alpha n)$ and its moments. This allowed him to apply the circle method to his problem, which is considered a binary problem. Using the upper bounds, we also study divisor sums over $[X,X+H]^{\ast}.$ Note that sums over powers of divisor functions over $[X,X+H]^{\ast}$ can be useful for studying several problems with restricted digits, such as Goldbach-type problems with restricted digits. For other arithmetic functions over integers with restricted digits, see \cite{Nath}.

\begin{Th}\label{main} Let $g$ be a  sufficiently large positive number, and let $|[1.H]^{\ast}|= \left(X/(\log X)^{3}\right)^{1/2\lambda} $ where $\lambda=1-\left(\log (\log g +1)/\log (g-1)\right).$  Then

\begin{equation}\begin{split} \sum_{ n \in [X,X+H]^{\ast}} d_{2}(n) &= 
\sum_{q \leq P} \int_{|\beta|<
\frac{1}{qQ}} \int_{X}^{X+H} p_{2,q}^{\ast}(x) e(\beta x) \sum_{n \in [X,X+H]^{\ast}} c_{q}(n) e(-n\beta) dx d\beta 
\\& + O \left( |[X,X+H]^{\ast}|(\log X)^{5/2}\right) \end{split}
\end{equation} as  $X \rightarrow \infty,$ 
where \begin{equation}\label{p2q}p_{2,q}^{\ast}(x):=\frac{1}{q}\left( \log x+ 2\gamma    -2\log q\right),\end{equation}
$$c_{q}(n) = \sum_{1 \leq a \leq q \atop (a,q)=1 } e\left(\frac{an}{q}\right),$$ and
$P=X^{1/4-\eta}, Q=P^{2}$ for some sufficiently small $\eta>0.$
\end{Th}
We couldn't obtain the main term due to a lack of understanding of the following double sums
$$\sum_{n \in [X,X+H]^{\ast}} \sum_{q<P} \frac{c_{q}(n)}{q}, \sum_{n \in [X,X+H]^{\ast}} \sum_{q<P} \frac{|c_{q}(n)|}{q}.$$ However, the size of the error term allows us to get an upper bound. By considering the number of multiples of $q \leq P$ in $[X,X+H]^{\ast},$ we can obtain an upper bound with a few extra power of logarithm (see Corollary \ref{Cor}). 
In \cite{Konyagin}, it is proved that there exist effectively computable positive constants \( c_5 = c_5(g) > 0 \) and \( c_6 = c_6(g) > 0 \) such that if \( P \leq N^{c_4} \), then
\begin{equation}\begin{split}\label{kon}
&\sum_{q \leq P ,(q, g) = 1} q \max_{a \, (\bmod q)} \left| \left|\left\{ n \in [X,2X]^{\ast} : n \equiv a \, (\bmod q) \right\}\right| - \frac{|[X,2X]^{\ast}|}{q} \right|
\\& \leq 
 c_5 |[X,2X]^{\ast}| P \exp \left( -c_6 \sqrt{\log X} \right),
\end{split}\end{equation}
where \( c_4 = \frac{\log 2}{4 \log g} \).
Although this provides a nice upper bound for the error term, it does not allow us to bound the main term in the above theorem, due to the condition on the size of $P.$ If we are only concerned with bounding the number of multiples of \( q \leq P \) among \( [X,X+H]^{\ast} \), we just need to estimate the correct order of magnitude of the number of multiples. Therefore, we use the following result, which gives us a better choice for \( P \).

\begin{Lemma}\label{Mau1} Let $g$ be a sufficiently large natural number.
For every \( A > 0 \), there exists \( B>0 \) such that
\begin{equation}\label{Mau} 
\sum_{\substack{q \leq |[1,X]^{\ast}|^{\lambda} (\log X)^{-B} \\ (q, g(g-1)) = 1}} \max_{a \, (\bmod q)} \left|\left| \left\{ n \in [1,X]^{\ast} : n \equiv a \, (\bmod q) \right\}\right| - \frac{|[1,X]^{\ast}|}{q} \right| 
\ll_{g} |[1,X]^{\ast}| (\log X)^{-A}
\end{equation}
where $\lambda$ is a positive constant such that 
 $$\int_{\alpha \in [0,1]} \left| \sum_{n \in [1,X]^{\ast}} e(n\alpha) \right| d\alpha \ll |[1,X]^{\ast}|^{1-\lambda}.$$
\end{Lemma}
\begin{proof} This is a special case of the results in \cite{DARTYGE2001230}.
\end{proof}
\begin{Rem} 
Since we set $X$ as a power of $g$ and $b$ is neither $0$ nor $1$, we have 
$1_{(X,X+H]^{\ast}}(n)= 1_{(0,H]^{\ast}}(n)$ where $1_{E}(n)$ denotes the indicator function of the set $E.$
Due to the large size of $\lambda,$ some elementary approaches give us the asymptotic results for some cases. For example, when considering the interval $[X,2X],$ 
$$\sum_{ n \in [X,2X]^{\ast}} d_{2}(n) = 
2 \sum_{m<X^{1 / 2}} \sum_{\substack{n<X \\ m \mid n}} 1_{[X,2X]^{\ast}}(n)+O(\sqrt{X})$$ and by applying Lemma \ref{Mau1}, we can get asymptotics. By following the argument in the proof of \cite[Theorem 2]{DARTYGE2001230}, it is straightforward to show that, as long as the size $|[1,H]^{\ast}|^{\lambda}$ is greater than $X^{1/2}(\log X)^{-1},$one can get the same upper bound in Corollary \ref{Cor}. One of the key features of our approach is that we only need to consider $q$ smaller than $P \ll X^{1/4}.$ Also, our approach might be useful for studying variations of other arithmetic functions that have a Voronoi-type formula.
\end{Rem}

Using Lemma \ref{Mau1} with Theorem \ref{main}, we can prove the following corollaries.
\begin{Cor}\label{Cor}
 Assume the conditions of Theorem \ref{main}. Then
\[
\sum_{n \in [X,X+H]^{\ast} } d_2(n) \ll d_{2}(g(g-1)) |[X,X+H]^{\ast}| (\log X)^3
\] as  $X \rightarrow \infty.$ 
\end{Cor}

\begin{proof}

Let \( L_q = \frac{2\pi}{qQ} \). By separating the integral in Theorem \ref{main}, we get
\[
\frac{1}{\pi}\int_X^{X+H} \frac{p_{2,q}^{\ast}(x)}{x-n} \sin \left( L_q (x - n) \right) dx \ll \left| \int_{L_q(X - n)}^{-L_{q}} + \int_{L_{q}}^{L_q(X +H - n)} p_{2,q}^{\ast}\left( n + \frac{x}{L_q} \right) \frac{\sin x}{x} dx \right| + \frac{\log X}{qQ}.
\]
Using the fact that
\[
\left| \int_0^x \frac{\sin y}{y} dy - \frac{\pi}{2} \right| \ll \frac{1}{|x|}
\]
for large \( x \neq 0 \), 
and 
\[
\frac{\sin x}{x} = 1 + o(x^{2})
\]
for small $x,$ 
we have
$$\int_{0}^{x} \frac{ \sin y}{y}dy \ll \min \left(\frac{\pi}{2}, |x|\right).$$
Also noting that $\left( p_{2,q}^{\ast}(x)\right)' = \frac{1}{qx} .$
By applying integration by parts, we have
\begin{equation}\label{2inte}\left| \int_{L_q(X - n)}^{-L_q}  p_{2,q}^{\ast}\left( n +\frac{x}{L_q} \right) \frac{\sin x}{x} dx \right| \ll \frac{\log X}{q} + \left|\int_{L_{q}(X-n)}^{-L_{q}} \frac{1}{qL_{q}\left(n+x/L_{q}\right)} \min \left(\frac{\pi}{2}, |x|\right) dx\right|
.\end{equation}
Since $L_{q} \leq |L_{q}(X-n)| \ll \frac{X}{qQ},$
$qL_{q}\left(n+x/L_{q}\right) \gg \frac{X}{Q}$ when $x \in [L_{q}(X-n), -L_{q}].$ Therefore, the integral of the right-hand side of $\eqref{2inte}$ is bounded by 
$$O\left(\frac{Q}{X}+\frac{1}{q}\right).$$
Hence, the left-hand side of \eqref{2inte} is bounded by $(\log X)/q.$
By the similar argument as the above,
we have
$$\left| \int_{L_q}^{L_q(X+H-n)}  p_{2,q}^{\ast}\left( n +\frac{x}{L_q} \right) \frac{\sin x}{x} dx \right| \ll \frac{\log X}{q},$$ and 
\[
\left| \int_{L_q(X - n)}^{-L_q} + \int_{L_q}^{L_q(X+H - n)} p_{2,q}^{\ast}\left( n +\frac{x}{L_q} \right) \frac{\sin x}{x} dx \right| \ll \frac{\log X}{q}.
\]
Thus, the entire contribution from the main term is bounded by
\[
\log X \sum_{q \leq P} \sum_{n \in [X,X+H]^{\ast}} \frac{|c_q(n)|}{q}.
\]
Using the crude bound \( |c_q(n)| \leq (n, q) \), we have 

\begin{equation}\label{on}
\begin{split}
\sum_{q \leq P} \sum_{n \in [X,X+H]^{\ast}} \frac{|c_q(n)|}{q} 
&\ll \sum_{a|g(g-1)} \sum_{q \leq P \atop (q, g(g-1)) = a} \frac{1}{q} \sum_{l|q} l \sum_{n \in [X,X+H]^{\ast}  \atop  (n, q) = l} 1 \\
&\ll \sum_{a|g(g-1)} \frac{1}{a} \sum_{q \leq P/a  \atop  (q, g(g-1) = 1} \frac{1}{q} \sum_{l|aq} l \sum_{n \in [X,X+H]^{\ast}  \atop  l|n} 1 \\
&\ll \sum_{a|g(g-1)} \sum_{l \leq aP} \sum_{n \in [X,X+H]^{\ast}  \atop  l|n}  \sum_{lq_1 \leq P  \atop  (lq_1, g(g-1)) = a} \frac{1}{q_1} \\
&\ll \log X \sum_{a|g(g-1)} \sum_{l \leq aP} \sum_{n \in [X,X+H]^{\ast} \atop  l|n} 1.
\end{split}
\end{equation}

By applying \eqref{Mau}, we see that 
\[
\sum_{l \leq aP} \sum_{n \in [X,X+H]^{\ast} \atop  l|n} 1 \ll_{A} |[X,X+H]^{\ast}| \left(\sum_{l \leq aP} \frac{1}{l}+ (\log X)^{-A} \right) \ll |[X,X+H]^{\ast}| \log X.
\]
Note that the coprimality condition in Lemma \ref{Mau1} can be easily resolved.
Thus, the proof is complete.
\end{proof}

\subsection{Sketch of the Proof} The argument starts with the standard Hardy-Littlewood circle method. 
Due to the substantial savings in $\sum_{n\in [X,X+H]^{\ast}} e(\alpha n ),$ we only need to obtain some savings from 
$\sum_{n\in [X,X+H]} d_{2}(n)e(n\alpha).$  Next, we examine the supremum of $\sum_{n \in [X, X+H]} d_{2}(n)e(n\alpha)$ as $\alpha$ lies on the minor arcs. A slight modification of the results in \cite{Julita} provides sufficiently small upper bounds for this supremum.
\section{Notations} From now on, we assume that $\eta>0$ is sufficiently small constant. 
 $\varepsilon$ may vary from line to line and we use the following notations:
\[
S_{2}(\alpha;X,X+H) := \sum_{n=X}^{X+H} d_{2}(n) e(n\alpha), \quad S_{2}(\alpha;x):= \sum_{n=1}^{x} d_{2}(n) e(n\alpha),
\]
\[ S_{[X,X+H]^{\ast}}(\alpha) := \sum_{n \in [X,X+H]^{\ast}} e(n\alpha), \quad F_{[X,X+H]}(\alpha) := |[X,X+H]^{\ast}|^{-1} \left| \sum_{n \in [X,X+H]^{\ast}} e(n\alpha) \right|,\]
\[
Q = X^{1/2 - 2\eta}, \quad P = X^{1/4 - \eta}, \quad
\mathcal{M} := \bigcup_{q \leq P} \bigcup_{\substack{1 < a < q \\ (a, q) = 1}} \left( \frac{a}{q} - \frac{1}{qQ}, \frac{a}{q} + \frac{1}{qQ} \right), \quad m= [0,1] \cap \mathcal{M}^{c}.
\]

\section{Lemma}
In this section, we provide various upper bounds of $S_{2}(\alpha;X,X+H).$ In \cite{Julita}, Jutila generalized the classical Voronoi summation formula to an additively twisted version. One of the crucial results he obtained in that paper is a generalization of Wilton's approximate functional equation (see Lemma \ref{ju}).
We essentially follow the arguments in \cite{Julita}.  
%\begin{Lemma}[The Piltz Divisor Problem]
%\begin{equation}
%S_{k}(0) = X p_{k}(\log X) + O(X^{1-c_{k}})
%\end{equation}
%for some $c_{k} \geq 1/k$
%where $p_{k}(x)$ is a polynomial of degree $k-1$. Under the Lindelof hypothesis, the supremum of $c_{k}=1/2 + 1/2k.$
%\end{Lemma}

%\begin{Lemma}[Variance of the Error Term]
%\[
%\int_{X}^{2X} |\Delta_{k}(x)|^{2} \, dx \ll X^{1+\alpha_{k}}
%\]
%where $\alpha_{k}=1, 2/3$ for $k=2,3,$ respectively. For $k>3, $ it is known that  $\alpha_{k} \geq 1-1/k.$
%Under the Lindelof hypothesis, the infimum of $\alpha_{k}=1-1/k$
%\end{Lemma}
\begin{Lemma}\label{ju}
Let \( q \leq X^{1/2} \) and \( 0 < |\beta| < \frac{1}{q^2} \). If \( q|\beta| \gg \frac{1}{X^{1/2}} \), then
\begin{equation}\begin{split}\label{ju1}
S_2\left( \frac{a}{q} + \beta; X,X+H \right) &= (q |\beta|)^{-1} \sum_{n \in [(q\beta)^2 X, (q\beta)^2(X+H)]} d_2(n) e\left( \left(\frac{-\bar{a}}{q} - \frac{1}{q^2 \beta}\right) n \right) 
\\&+O\left(X^{1/2}\log X\right),
%O_{\varepsilon}\left(|q\beta|H  +qQ^{-1/2}X^{1/4}+q^{4}|\beta|^{7/2}X^{1/2+\varepsilon}\right),
\end{split}\end{equation}
and
\begin{equation}\begin{split}\label{ju2}
S_2\left(  \frac{a}{q};X,X+H\right) &= (X+H)p_{2,q}(X+H) - Xp_{2,q}(X) \\&\quad +  O_{\varepsilon}\left(q^{1/2}HX^{-1/2}(\log X)^{2}+ qX^{\varepsilon}\right)
\end{split}\end{equation}
where $p_{2,q}(x):= \frac{1}{q}\left( \log x + 2\gamma -1 -2\log q\right)$. 
Note that
$$\int_{X}^{X+H} p_{2,q}^{\ast}(x)dx = (X+H)p_{2,q}(X+H) - Xp_{2,q}(X).$$
\end{Lemma}
\begin{proof}
By using \cite[Theorem 5, (1.12)]{Julita}, the first equation is easily proved.

Now, let us consider the second equation. 
When $x, N \in [X,2X],$
$$S_{2}\left(\frac{a}{q};x\right) = xp_{2,q}(x)+ \Delta(a/q;x) + O\left(q (\log 2q)^2{}\right).$$ where  

$$
\begin{aligned}
\Delta(a/q; x)=(\pi \sqrt{2})^{-1} q^{1 / 2} x^{1 / 4} & \sum_{n \leq N} d(n) e(-n \bar{a}/q) n^{-3 / 4} \cos (4 \pi \sqrt{n x} / q-\pi / 4) \\
& +O\left(q X^{\epsilon}\right) ,
\end{aligned}
$$
 (see \cite[(1.4)]{Julita}, \cite[Section 1.8, Theorem 1.1]{Bookjutila}). 
Therefore, it is easy to see that 
\begin{equation}\begin{split}S_{2}\left(\frac{a}{q};X+H\right) - S_{2}\left(\frac{a}{q};X\right)&= (X+H)p_{2,q}(X+H) - Xp_{2,q}(X) \\&+ O_{\varepsilon}\left(q^{1/2}HX^{-1/2}(\log X)^{2}+ qX^{\varepsilon}\right).\end{split}\nonumber\end{equation}.
\end{proof}

%Note that when $x \in [X,2X]$
%$$S_{2}\left(\frac{a}{q},x\right) = x p_{2,q}(x)+ O(q (\log 2q)^{2} + q^{2/3} X^{1/3+\varepsilon})$$ where  
%$p_{2,q}(x):= \frac{1}{q}\left( \log x + 2\gamma -1 -2\log q\right)$ (see \cite[(1.4)]{Julita}. This comes from the Laurent expansion 
%$$E\left(s, \frac{a}{q}\right)=q^{-1}(s-1)^{-2}+q^{-1}(2 \gamma-2 \log q)(s-1)^{-1}+\cdots,$$
%and using the Perron's formula $D\left(x, \frac{a}{q}\right)=\frac{1}{2 \pi i} \int_{1+\varepsilon-i T}^{1+\varepsilon+i T} E\left(s, \frac{a}{q}\right) x^s s^{-1} d s+O\left(x^{1+\varepsilon} T^{-1}\right)$
%(see \cite{Julita}[(2.2), (2.4)] ).

\begin{Lemma}[\cite{Julita}, Theroem 2]\label{julita2}
For $X \geq q  $ we have 
\begin{equation}\begin{split} &\int_{1}^{X} \left|S_{2}(\frac{a}{q},x)  -\frac{x}{q}(\log x +2\gamma -1 -2\log q)\right|^{2} d x \\&=\left(6 \pi^2\right)^{-1} \zeta^4\left(\frac{3}{2}\right) \zeta(3)^{-1} q X^{\frac{3}{2}}+O_{\varepsilon}\left(q^2 X^{1+\varepsilon} + Xq^{2}\log^{4} (2q)+q^{\frac{3}{2}} X^{\frac{5}{4}+\varepsilon}\right).\end{split} \nonumber \end{equation}
\end{Lemma}
\begin{proof}
This comes from \cite[Theorem 2]{Julita} with the bound $E(0,a/q) \ll q (\log 2q)^{2}.$
\end{proof}
The following lemma shows that, on average, $F_{[X,X+H]}(\alpha)$ is of size $$\left|[X,X+H]^{\ast}\right|^{-1+ \frac{\log \left((\log g)+1\right)}{\log (g-1)}}.$$
\begin{Lemma}\label{may}  Let $g$ be sufficiently large. Then 
\[
\int_{[0,1]} F_{[X,X+H]}(\alpha) \, d\alpha \ll \left|[X,X+H]^{\ast}\right|^{-1+ \frac{\log \left((\log g)+1\right)}{\log (g-1)}}
\]
\end{Lemma}
\begin{proof} Using the fact that 
$$1_{(X,X+H]^{\ast}}(n)= 1_{(0,H]^{\ast}}(n)$$ along with 
\cite[Lemma 10.3, Section 16]{Jmay}, the proof is completed.
\end{proof}
\begin{Lemma}\label{35}
Let $|\beta| < T$ for some $0 < T < 1,$ and let $1 \leq q \leq X^{1/2}.$ Then

\begin{equation}\label{integeralbyparts}\begin{split}
S_{2}(\frac{a}{q}+\beta;X,X+H) &= \int_{X}^{X+H} p_{2,q}^{\ast}( x) e(\beta x) \, dx \\&+ O_{\varepsilon}\left
(qX^{\varepsilon} + q^{1/2}H X^{-1/2} (\log X)^{2}+Tq^{3/4}X^{5/8+\varepsilon}H^{1/2}\right).
\end{split}\end{equation}
\end{Lemma}

\begin{proof}
By integration by parts (see \cite[Lemma 2.1]{MRT1}) along with Lemmas 3.1,3.2, we have 
\begin{equation}\begin{split}
&S_{2}(\frac{a}{q}+\beta;X,X+H)  - \int_{X}^{X+H} p_{2,q}^{\ast}(x) e(\beta x) \, dx \\& \ll_{\varepsilon} q^{1/2}HX^{-1/2}(\log X)^{2}+ qX^{\varepsilon} + \left| T \left(\int_{X}^{X+H} |\Delta(a/q;x)|^{2}dx\right)^{1/2} \left(\int_{X}^{X+H} 1 dx\right)^{1/2} \right|
%\\&\ll_{\varepsilon} q^{1/2}HX^{-1/2}(\log X)^{2}+ qX^{\varepsilon} + T\left(q^{1/2}HX^{1/4} + X^{1/2+\varepsilon}H^{1/2}q + X^{1/2}q (\log 2q)^{2}H^{1/2} + X^{5/8+\varepsilon}q^{3/4}H^{1/2}\right)
%\\&\ll_{\varepsilon} q^{1/2}HX^{-1/2}(\log X)^{2}+ qX^{\varepsilon} + Tq^{1/2}X^{5/4+\varepsilon}.
\end{split}\end{equation}
By using Lemma 3.2, it is easy to see that 
\begin{equation}\begin{split}
&\left| T\left(\int_{X}^{X+H} |\Delta(a/q;x)|^{2}dx\right)^{1/2} \left(\int_{X}^{X+H} 1 dx\right)^{1/2} \right|
\\&\ll T\left(q^{1/2}HX^{1/4} + X^{1/2+\varepsilon}H^{1/2}q + X^{1/2}q (\log 2q)^{2}H^{1/2} + X^{5/8+\varepsilon}q^{3/4}H^{1/2}\right),
\end{split}\end{equation} 
which is bounded by 
$TX^{5/8+\varepsilon}q^{3/4}H^{1/2}.$\end{proof}
\section{Propositions}

\begin{Prop}\label{sup} 
When $
\alpha \in (\frac{a}{q}-\frac{1}{qQ}, \frac{a}{q}+\frac{1}{qQ})$ for some $1 \leq a \leq q\leq P, (a,q)=1,$ we have 

\begin{equation}\label{majorerror}\left|S_{2}(\alpha;X,X+H) - \int_{X}^{X+H} p_{2,q}^{\ast}(x) e(\alpha x) dx \right|  \ll
_{\varepsilon} X^{3/8+\varepsilon+2\eta}.
\end{equation}
And when $\alpha \in m,$ we have 
$$\sup_{\alpha \in m} S_{2}(\alpha)\ll_{\varepsilon} X^{1/2-\eta+\varepsilon}+X^{1/2}\log X.$$
\end{Prop}
\begin{proof}

If $\alpha= \frac{a}{q}$ for some  $1 \leq a \leq q\leq P, (a,q)=1,$ then by \eqref{ju2},
it is bounded by $P^{1/2}HX^{-1/2}(\log X)^2 + P X^{\varepsilon}.$ 
If not, then by using \eqref{integeralbyparts}, 
$$\left|S_{2}(\alpha;X,X+H) - \int_{X}^{X+H} p_{2,q}^{\ast}(x) e(\alpha x) dx \right|\ll_{\varepsilon} PX^{\varepsilon} +Q^{-1}X^{5/8+\varepsilon}H^{1/2}$$ 
(since $H < X^{1/2}, $ the second part of the error term can be ignored). 

Now, let us consider the second inequality. By the Dirichlet theorem, for any $\alpha \in m,$ there exists a fraction $\frac{a}{q}$ such that $1 \leq a \leq q, (a,q)=1,$ $P< q \leq Q,$ and $\left|\alpha-\frac{a}{q}\right| \leq \frac{1}{qQ}.$

First, let $\beta=|\alpha - \frac{a}{q}| \gg \frac{1}{q^{2/3}X^{1/2}}$ for some 
$1 \leq a \leq q, (a,q)=1,$ $P< q \leq Q.$ Then by using \eqref{ju1}, 
\begin{equation}\begin{split}S_{2}(\alpha;X,X+H) &\leq |q\beta| \sum_{n \in [|q\beta|^{2}X, |q\beta|^{2}(X+H)]} d_{2}(n) + O(X^{1/2}\log X) \\&\ll_{\varepsilon} \frac{H\log X}{Q} +X^{1/2}\log X.\end{split}\end{equation}
 Let $|\alpha - \frac{a}{q}| \ll \frac{1}{q^{2/3}X^{1/2}}.$ Then by \eqref{integeralbyparts}, we have  
     \begin{equation}\begin{split}S_{2}(\alpha;X,X+H) &\ll_{\varepsilon} \frac{H\log X}{P}   + QX^{\varepsilon} + Q^{1/2}HX^{-1/2+\varepsilon}+ Q^{1/12}X^{3/8+\varepsilon-\eta/4}.\end{split}\end{equation}
 Using the fact that $P^{2}=Q$ and $ \log Q < (\log X)/2,$ the proof is completed.
\end{proof}

\begin{Prop}\label{Prop1} Let $T=\frac{1}{qQ}.$ Then
\begin{equation}\label{majorerror}\begin{split}
\sum_{n \in [X,X+H]^{\ast}} & \int_{\alpha \in \mathcal{M}} S_{2}(\alpha;X,X+H) e(-n \alpha)  \, d\alpha
\\&= \sum_{q \leq P} \sum_{1 \leq a <q \atop (a,q)=1} \sum_{n \in [X,X+H]^{\ast}} \frac{1}{\pi} \int_{X}^{X+H} \frac{p_{2,q}^{\ast}(x)}{x-n} \sin (2\pi T(x-n)) \, dx + O\left(X^{3/8+2\eta+\varepsilon}\right).
\end{split}\end{equation}

\end{Prop}
\begin{proof} By using Lemma \ref{35}, we only need to consider the error term contributions.
The contribution from the error terms is bounded by 
$$ \sup_{\alpha \in \mathcal{M}} \left(\left|S_{2}(\alpha;X,X+H) - \int_{X}^{X+H} p_{2,q}^{\ast}(x) e(\alpha x) dx \right|\right)\int_{\alpha \in \mathcal{M}}  |[X,X+H]^{\ast}| F_{[X,X+H]}(\alpha) d \alpha. $$
By using the bounds of the supremum in Proposition \ref{sup} and the exponential sum bound in Lemma \ref{may}, the above term is bounded by 
$$O\left(X^{3/8+2\eta+\varepsilon}\right).$$

%\[
%\int_{|\beta| < T} S_{k}(\beta) e(-\beta n) \, d\beta = \int_{|\beta| < T} \int_{X}^{2X} p_{k}(\log x) e(\beta x) e(-\beta n) \, d\beta + O\left(X^{1+\frac{\alpha_{k}}{2}} T^{2} + X^{1-c_{k}} T\right).
%\]

%And we have 
%\[
%\int_{|\beta| < T} e(\beta (x-n)) \, d\beta = \frac{e^{2\pi iT (x-n)} - e^{-2 \pi iT(x-n)}}{2\pi i (x-n)} = \frac{\sin(2 \pi T(x-n))}{\pi (x-n)},
%\]
%when $x \neq n$, and $2T$ if $x = n$.

%Since
%\[
%\lim_{x \to n} \frac{\sin(2 \pi T(x-n))}{\pi (x-n)} = 2T,
%\]
%we can ignore the point $x = n$ and obtain the main term.

\end{proof}

\begin{Prop}\label{Prop2}

\[
\sum_{n \in [X,X+H]^{\ast}} \int_{\alpha \in m} S_{2}(\alpha;X,X+H) e(-\alpha n) \, d\alpha \ll_{\varepsilon} X^{1/2 -\eta + \varepsilon}+\left|[1,H]^{\ast}\right|(\log X)^{5/2}.
\]

\end{Prop}
\begin{proof}
The contribution from the minor arc is bounded by 
$$ \sup_{\alpha \in m}|S_{2}(\alpha;X,X+H)| \int_{\alpha \in m}  |[X,X+H]^{\ast}| F_{[X,X+H]}(\alpha) d \alpha.  $$
By using the bounds of the supremum in Proposition \ref{sup}, the proof is completed.
\end{proof}
\subsection{Proof of Theorem 1.1}
By combining Propositions \ref{Prop1} and \ref{Prop2}, Theorem \ref{main} is proved.

\subsection*{Acknowledgements} The author would like to thank Professor Micah Milinovich for suggesting various short-interval results for the divisor sum and results relevant to the main term, as well as Jaime Hernandez Palacios for useful discussions. The author also thanks Kunjakanan Nath for informing him about the results relevant to Lemma 1.2 and for providing the elementary proof for the case of long intervals.
%\section{short interval}
%For the short interval version, we use the following notations:
%\[
%x(H)^{\ast} := [X,X+H]^{\ast} \cap \left( [x, x+H] \cup [2x, 2x+2H] \right),
%\]
%\[
%T_{2}(x, \alpha) := \sum_{n=x}^{x+H} d_{2}(n) e(n\alpha) + \sum_{n=2x}^{2x+2H} d_{2}(n) e(n\alpha), \quad
%E_{x, H}(\alpha) := \left| \sum_{n \in x(H)^{\ast}} e(n\alpha) \right|.
%\]

%Note that:
%\[
%E_{x, H}(\alpha) \ll |[X,X+H]^{\ast}| F_{X}(\alpha).
%\]
%By Lemma 3.1, 
%$$S_{k}(\beta) = \int_{X}^{2X} p_{k}( \log x) e(\beta x) dx + O\left(X^{1+\frac{\alpha_{k}}{2}}T + X^{1-c_{k}}\right).$$ 
%By integrating by parts,
%\[
%\int_{X}^{2X} p_{k}(\log x) e(\beta x) \, dx \ll \frac{(\log X)^{k-1}}{|\beta|}.
%\]
%Thus,
%\[
%\sup_{T < |\beta| < 1} |S_{k}(\beta)| \ll X^{1+\alpha_{k}/2+\varepsilon}.
%\]

%Therefore,
%\[
%\sum_{n \in [X,X+H]^{\ast}} \int_{T < |\beta| < 1} S_{k}(\beta) e(-\beta n) \, d\beta \ll \sup_{T < |\beta| < 1} |S_{k}(\beta)| \left| [X,X+H]^{\ast} \right| \int_{[0,1]} F_{[X,X+H]^{\ast}}(\alpha) \, d\alpha
%\]
%\[
%\ll X^{1 + \alpha_{2}/2 + \varepsilon - 1} \left| [X,X+H]^{\ast} \right|.
%\]

\bibliographystyle{plain}   % this means that the order of references
			    % is dtermined by the order in which the
			    % \cite and \nocite commands appear
\bibliography{over}  % list here all the bibliogres that
			     % you need. 
\end{document}